\documentclass[a4paper]{mjm}
\usepackage[dvipsnames]{color}
\usepackage{amsmath,amsfonts}
\usepackage{latexsym}
\usepackage[latin1]{inputenc}

\DeclareMathOperator{\ind}{ind} 
\DeclareMathOperator{\im}{Im}
\newcommand*{\boundary}{\partial}

\title{Families index for\\Boutet de Monvel operators}
\headtitle{Families of Boutet de Monvel operators}
\author{Severino T. Melo \and Thomas
  Schick \and Elmar Schrohe}
\headauthor{S. Melo, T. Schick, E. Schrohe}

\address{
Severino T. Melo\\ Instituto de  Matem{\'a}tica e Estat{\'\i}stica,
Universidade de S{\~a}o Paulo\\
 Rua do Mat{\~a}o 1010,
05508-090 S{\~a}o Paulo, Brazil|}
 \email{toscano@ime.usp.br}
\urladdr{http://www.ime.usp.br/~toscano}

 \address{Thomas Schick\\ Mathematisches Institut, Georg-August-Universit\"at
     G\"ottingen\\  Bunsenstr.\ 3-5, 
37073 G{\"o}ttingen, Germany}
\email{schick@uni-math.gwdg.de}
  \urladdr{http://www.uni-math.gwdg.de/schick}

\address{Elmar Schrohe\\ Institut f\"ur Analysis, Leibniz Universit\"at Hannover\\ Welfengarten 1, 
30167 Hannover, Germany} \email{schrohe@math.uni-hannover.de}
\urladdr{http://www.analysis.uni-hannover.de/~schrohe/}

\IfFileExists{META}{\input{META}}{\newcommand{\mjmfirstpage}{99999}\newcommand{\mjmlastpage}{99999}}
\usepackage{mjm}

\newtheorem{thm}{Theorem}
\newtheorem{theorem}[thm]{Theorem}
\newtheorem{pro}[thm]{Proposition}
\newtheorem{proposition}[thm]{Proposition}

\newtheorem{corollary}[thm]{Corollary}
\newtheorem{df}[thm]{Definition}
\newtheorem{definition}[thm]{Definition}
\theoremstyle{definition}
\newtheorem{remark}[thm]{Remark}

\begin{document}
\newcommand{\iso}{\cong}
\newcommand{\A}{{\mathcal A}}
\newcommand{\ac}{{\mathfrak A}}
\newcommand{\Ps}{{\mathcal P}}
\newcommand{\pc}{{\mathfrak P}}
\newcommand{\kc}{{\mathfrak K}}
\newcommand{\hc}{{\mathfrak H}}
\newcommand{\R}{{\mathbb R}}
\newcommand{\reals}{{\mathbb R}}
\newcommand{\C}{{\mathbb C}}
\newcommand{\complexs}{{\mathbb C}}
\newcommand{\N}{{\mathbb N}}
\newcommand{\naturals}{{\mathbb N}}
\newcommand{\Z}{{\mathbb Z}}
\newcommand{\rn}{{\mathbb R}^{n}}
\newcommand{\gp}{G_{\oplus}}
\newcommand{\diff}{\mbox{Diff}}
\newcommand{\op}{operator}
\newcommand{\ops}{operators}
\newcommand{\psd}{pseudo\-dif\-fer\-en\-tial}
\newcommand{\Psd}{Pseudo\-dif\-fer\-en\-tial}
\newcommand{\cqd}{\hfill$\Box$}
\newcommand{\h}{{\mathcal H}}
\newcommand{\pf}{{\em Proof}: }
\newcommand{\asim}{\!\!\sim}


\begin{abstract}
We define the analytical and the topological indices for continuous 
families of operators in the C$^*$-closure of the Boutet de Monvel
algebra. Using techniques of C$^*$-algebra K-theory and the
Atiyah-Singer theorem for families of elliptic operators on a closed manifold, 
we prove that these two indices coincide.
\end{abstract}
\maketitle
\begin{center}{\footnotesize 2010 Mathematics Subject Classification: 19K56, 46L80, 58J32}
\end{center}

\thispagestyle{empty}
\section*{Introduction}
Boutet de Monvel's calculus \cite{B} provides a pseudodifferential framework 
which encompasses the classical differential boundary value problems. 
In an extension of the concept of Lopatinski and Shapiro, it associates 
to each operator two  symbols: 
a pseudodifferential principal symbol, which is a bundle homomorphism, and  
an operator-valued boundary symbol.
Ellipticity requires the invertibility of both. In this case, the calculus allows the 
construction of a parametrix. 
If the underlying manifold is compact, elliptic elements define Fredholm operators,
and the parametrices are Fredholm inverses. 
Boutet de Monvel showed how then the index can be computed in 
topological terms. 
The crucial observation is that elliptic operators can be mapped to compactly 
supported K-theory classes on the cotangent bundle over the interior of the 
manifold. 
The topological index map, applied to this class, then furnishes an integer which 
is equal to the index of the operator.

For the construction of the above map, Boutet de Monvel combined 
operator homotopies and classical (vector bundle) K-theory in a
very refined way. It therefore came as a surprise that this map -- 
which is neither obvious nor trivial -- 
can also be obtained as a composition of various standard 
maps in K-theory for C$^*$-algebras  
-- which was not yet available when \cite{B} was written. 
In fact, it turns out to be basically sufficient to have 
a precise understanding of the short exact sequence induced by the 
boundary symbol map, \cite{MSS}, see also \cite{MNS}. 

In the spirit of the classical result of Atiyah and Singer \cite{AS} 
we introduce and consider in this article {\em families} of operators in 
Boutet de Monvel's calculus, an issue that has not been addressed in \cite{B}.

More specifically, we consider a compact manifold $X$ with boundary and then a
fiber bundle $Z\to Y$ with fiber $X$ over a compact Hausdorff space $Y$. We
are then studying fiberwise (elliptic) Boutet de Monvel operators, depending
continuously on $y\in Y$.
In order to be able to use the powerful tools of C$^*$-algebra K-theory 
we define such an operator family $A$ over $Y$
as a continuous section	 of a bundle of C$^*$-algebras over $Y$, 
a concept which is slightly more general than that of Atiyah and
Singer,  who equip the set of operators with a Fr\'echet-space topology.
In fact, restricted to the case without boundary,  
our algebra of continuous families $\ac$ contains that of \cite{AS} 
as a dense subalgebra.

While the analytic index $\ind_a(A)$ of such an elliptic family $A$ as an element of $K(Y)$ 
is easily defined following Atiyah \cite{A} and J\"anich \cite{J}, 
cf.~Definition \ref{ani} below,
it is less obvious how to obtain the topological description.
Similar to Boutet de Monvel's approach, the essential step is the
construction of a map which associates to an elliptic family an element of 
the compactly supported K-theory of the total space of the bundle of cotangent spaces over the interior of the underlying manifolds. 
We regard this map as a homomorphism defined on $K_1(\ac/\kc)$, 
where $\kc$ denotes the ideal of continuous families which have values in
compact operators. 
In its definition, we use a fact which builds upon an observation of Boutet 
de Monvel: There exists a natural subalgebra $\ac^\dagger$ of $\ac$ for which 
$K_*(\ac^\dagger/\kc)\cong K_*(\ac/\kc)$ so that each elliptic family $A$ in
$\ac$  
can be represented by a class $a\in K_1(\ac^\dagger/\kc)$. 
Moreover, $\ac^\dagger/\kc$ is commutative which allows us to make
the connection to classical (vector bundle) K-theory. Then $\ind_t(A)$ is
defined by applying the classical construction of the topological index to
$a$, compare Definition \ref{indt}.

Our main result is then that these two indices are equal. To
prove this, we reduce to the classical families index theorem of Atiyah and
Singer \cite{AS}. We assign in a canonical way to $A$
an index problem on a bundle of closed manifolds, namely the double of our 
original bundle of manifolds with boundary. We then show that this associated
family has the same analytic as well as topological index as $A$. 
In this step we make once more use of the isomorphism 
$K_1(\ac/\kc)\cong K_1(\ac^\dagger/\kc)$.

It is perhaps worth stressing that our index theorem does not use the Boutet de Monvel index theorem for boundary value problems, which can actually be obtained from ours by taking $Y$ equal to one point.
Taking the families index theorem for granted, Albin and Melrose 
derived a more refined formula for the Chern character of the index bundle 
in terms of symbolic data \cite[Theorem 3.8]{AM}. 

The paper is structured as follows: 
Section~\ref{sec:BdM_single} starts with a review of the Boutet de Monvel
calculus for a single manifold. We introduce the
C$^*$-algebra $\mathcal{A}$ of Boutet de Monvel operators of order and class
zero and the boundary symbol map $\gamma$.  
Section~\ref{sec1} gives  the technical introduction of operator families in Boutet de
Monvel's calculus over a compact Hausdorff space $Y$. 
We define them as the continuous sections into a bundle of operator algebras 
whose typical fiber is the C$^*$-algebra $\mathcal A$. 
In order to keep the exposition simple, we first treat the case where $E$ is trivial 
one-dimensional and $F=0$.
We introduce $\gamma$ as the fiberwise symbol map and extend the results on 
the kernel and image of $\gamma$ to the family situation.
 
While in the single operator case this was sufficient to compute the K-theory
of $\mathcal A/\mathcal K$, the situation is more complicated in the families case. 
In fact, an important ingredient in \cite{MSS}
is that fact that whenever $X$ is connected  
and $\partial X\not=\emptyset$ there exists a continuous section of $S^*X^\circ$. 
This is no longer true in the families case.
Instead, we prove in Theorem  \ref{Kth} the fact alluded to above: 
For $F=0$ we define $\ac^\dagger$ as the C$^*$-algebra 
generated by all sections whose 
pseudodifferential part is independent of the co-variable at the boundary 
and whose singular Green part vanishes. Then
$\ac^\dagger/\kc$ is commutative. Moreover, we use a
  Mayer-Vietoris argument to show that the inclusion map induces an 
isomorphism
\begin{equation}\label{K}
K_*(\mathfrak{A^\dagger /K})\cong K_*(\mathfrak{A/K}). 
\end{equation}

In Section~\ref{index} we study the index problem. 
Again, we confine ourselves first to the case of trivial 
one-dimensional bundles. We introduce the analytic and topological index 
and, as our main result, prove that the analytic and the
  topological index  are equal. To achieve this, we reduce with the help of a
  doubling procedure to the case of families of closed manifolds. This
  reduction is based
on the fact that we can use the isomorphism in \eqref{K} to represent any
element of  $K_1(\ac/\kc)$ as a $K_1$-class of $\mathfrak{A^\dagger /K}$.
In Section~\ref{sec:non-trivial_bundles} we finish by explaining the arguments needed for the general situation. 

Two appendices give technical details about the structure group of our
families and about the K\"unneth theorem we are using.

\section{Boutet de Monvel calculus for a single manifold}
\label{sec:BdM_single}

In this section, we introduce notation and recall the case of single 
operators. Details can be found in the monographs of 
Rempel and Schulze \cite{RS} and Grubb \cite{G} as well as in the 
short introduction \cite{S3}.

Let $X$ be a compact manifold of dimension $n$ with boundary $\partial X$
and interior $X^\circ$. We equip $X$ with a {\em collar} (i.e, a neighborhood 
$U$ of the boundary and a diffeomorphism $\delta\colon U\to\partial
X\times[0,1)$) 
which then induces the {\em boundary defining function $x_n=pr_{[0,1)}\circ \delta$}
The variables of $\partial X$ will be  denoted $x'$. 
The collar is used to provide the double $2X$ of $X$
with a (noncanonical) smooth structure.
Recall that $2X$ is the union of two copies $X^+$
and $X^-$ of $X$ quotiented by identification of the two copies of $\partial X$. 
 
An element in Boutet de Monvel's calculus is a matrix of operators
\begin{eqnarray}\label{eq.1}
A = \begin{pmatrix}P_++G&K\\T&S\end{pmatrix} \colon
\begin{array}{ccc}C^\infty(X,E_1)&& C^\infty(X,E_2)\\
\oplus&\longrightarrow &\oplus\\
C^\infty(\partial X,F_1)&& C^\infty(\partial X,F_2)\end{array},
\end{eqnarray}
acting between sections of vector bundles $E_1, E_2$ over $X$ and 
$F_1, F_2$ over $\partial X$. 
In this article we shall focus on the case of endomorphisms,
where $E_1=E_2=E$ and $F_1=F_2=F$.
For convenience, we choose a Riemannian
metric $g$ on $M$ and Hermitean metrics on $E,F$ to later obtain fixed Hilbert
spaces structures, although the results do not depend on these choices.
The operator $P_+$ in the upper left corner is a truncated 
pseudodifferential operator, derived from a (classical) pseudodifferential 
operator $P$ on $2X$. 
Given $u\in C^\infty(X,E)$, $P_+u$ is defined as the
composition  $r^+Pe^+u$. Here $e^+$ extends $u$ by zero to a function on $2X$, 
to which $P$ is  applied. The result then is restricted (via $r^+$) to $X$. 
In general it is not true that  $P_+u\in C^\infty(X,E)$. In order to ensure this, 
$P$ is required to satisfy the \emph{transmission condition}: 
If $p\sim\sum p_j$ is the asymptotic expansion of the local symbol $p$ of $P$ into terms $p_j(x,\xi)$, 
which are positively homogeneous of degree $j$ in $\xi$ 
one requires that, for $x_n=0$ and $\xi=(0,\pm1)$ one has
$D_{x}^\beta D^\alpha_\xi p_j(x',0, 0,1) = (-1)^{j-|\alpha|}
D_{x}^\beta D^\alpha_\xi p_j(x',0, 0,-1)$.
As for the remaining entries, 
$G$ is a singular Green operator, $T$ a trace operator, $K$ a 
potential operator, and $S$ a pseudodifferential operator on the 
boundary. 

Operators in Boutet de Monvel's calculus have an {\em order} and 
a {\em class} or {\em type}. There are invertible elements in the calculus
which allow us to reduce both, order and class, to zero. 
The operators then form a $*$-subalgebra
of the bounded operators on the Hilbert space 
$H:=L^2(X,E) \oplus L^2(\partial X, F)$.

\begin{df}\label{A}
Let $\A^\circ(E,F)$ denote the algebra of the 
(polyhomogeneous) Boutet de Monvel operators of order and class zero on
$H=L^2(X,E)\oplus L^2(\partial X,F)$,
endowed with its natural Fr\'echet topology,  
and $\A$ its C$^*$-closure in the algebra of all bounded operators on $H$.
We write $\A^\circ$ and $\A$ if $E=X\times\complexs$ is trivial 
one-dimensional and $F=0$.
\end{df}

Let $A\in \A^\circ(E,F)$ be given as in \eqref{eq.1}. For each entry
$P,S,G,T,K$ we have a symbol.
This is the usual one for $P$ and $S$, while 
$G$, $T$, and $K$ can be considered
as operator-valued pseudodifferential operators on $\partial X$ with classical
symbols in the sense of Schulze \cite{s91}.

These are defined as follows, see \cite{S3}: 
The principal pseudodifferential symbol $\sigma(A)$ of $A$ 
is the restriction of the principal symbol of $P$ to the cosphere bundle over $X$. 
In order to define the boundary principal symbol $\gamma(A)$
we first denote by $p^0$, $g^0$, $t^0$, $k^0$, and $s^0$ 
the principal symbols of $P$, $G$, $T$, $K$, and $S$, 
respectively. 
We let $E_{x',\xi'}^0$ be the pullback of $E|_{\{x_n=0\}}$ to 
the normal bundle of $X$,  lifted to $(x',\xi')\in S^*\partial X$. 
For fixed $(x',\xi')\in S^*\partial X$, $\xi_n\mapsto p^0(x',0,\xi',\xi_n)$ 
is a function on the conormal line in $(x',\xi')$, acting on $E^0_{x',\xi'}$.
It induces a truncated pseudodifferential operator  
\begin{equation*}
p^0(x',0,\xi',D_n)_+ = r^+p^0(x',0,\xi',D_n)e^+\colon
L^2(\R_{\ge0},E^0_{x',\xi'})
\to L^2(\R_{\ge0},E^0_{x',\xi'}).
\end{equation*}
In local coordinates near the boundary we then define the boundary principal
symbol $\gamma(A)(x',\xi')\colon L^2(\R_{\ge0},E^0_{x',\xi'})
  \oplus F_{x',\xi'}\to L^2(\R_{\ge0},E^0_{x',\xi'})\oplus 
F_{x',\xi'}$ by
\begin{equation}\label{sec:def_of_boundary_symb}
\gamma(A)(x',\xi'):=\begin{pmatrix}p^0(x',0,\xi',D_n)_+ +g^0(x',\xi',D_n)&k^0(x',\xi',D_n)\\
t^0(x',\xi',D_n)&s^0(x',\xi')\end{pmatrix}, 
\end{equation}
with $D_n$ indicating that we let the symbol act as an operator with respect to the variable $x_n$ only.
Note that the operator $g^0(x',\xi',D_n)$ is compact and that $k^0(x',\xi^\prime,D_n)$, 
$t^0(x',\xi',D_n)$ and $s^0(x',\xi')$ even have finite rank. 
The operator $p^0(x',0,\xi',D_n)_+$ on the other hand is a Toeplitz type 
operator; it will not be compact unless $p^0=0$. 

Denoting by $\mathcal K= \mathcal K(H)$ the ideal of compact operators on 
$\mathcal L(H)$, one has the following important estimate based on work by 
Gohberg \cite{Goh}, Seeley \cite{See} and Grubb-Geymonat \cite{GG},
see  \cite[2.3.4.4, Theorem 1]{RS} for a proof:
\begin{eqnarray}\label{normestimate}
\inf_{K\in\mathcal K} \|A+K\|= 
\max\{\|\sigma(A)\|_{\sup}, \|\gamma(A)\|_{\sup}\},
\end{eqnarray} 
where the sup-norms on the right hand side are over the cosphere bundles 
in $X$ and $\partial X$, respectively.
This estimate implies, in particular, that both symbols extend continuously to 
C$^*$-algebra homomorphisms defined on $\mathcal A(E,F)$. 
For fixed $(x',\xi')$ the range  $\{\gamma(A)(x',\xi')\mid A\in\mathcal A\}$ 
forms an algebra of Wiener-Hopf type operators. 

It also follows from this estimate that $\gamma$ vanishes on $\mathcal K$. 
Since the entries of $\gamma(A)(x^\prime,\xi^\prime)$  induced by $g^0$, $k^0$, $t^0$ and $s^0$
are (pointwise) compact while that induced by $p^0$ is not (unless $p^0=0$), 
we conclude that a Boutet de Monvel operator $A$ belongs to $\ker \gamma$  
if and only if $\sigma(A)$ vanishes at the boundary. 
Based on this observation (see \cite[Section~2]{MNS} for
details) one can show that $\sigma$ induces an isomorphism
\begin{equation}\label{eq:ker_gamma_single_mf}
 \ker \gamma/\mathcal K\cong C_0(S^*X^\circ).
\end{equation}

The K-theory of the range of $ \gamma$ was described in 
\cite[Section~3]{MNS}. 
Let ${\tt b}\colon C(\partial X)\to \im\gamma$ denote the C$^*$-homomorphism 
that maps $g$ to $\gamma(m(f))$, where $m(f)$ is the operator of multiplication 
by a function $f\in C(X)$ whose restriction to $\partial X$ equals $g$. 
Then {\tt b} induces a K-theory isomorphism.

\section{K-Theory of the families C$^*$-algebra}\label{sec1}

To simplify the exposition, we shall assume in this section that $E=X\times
\complexs$ is the trivial one-dimensional line bundle and $F=0$.
 
Let $\diff(X)$ denote the group of diffeomorphisms of $X$, equipped with its
usual Fr\'echet topology. Recall that $\delta\colon U\to
\boundary X\times [0,1)$ is the collar fixed at the beginning of Section
\ref{sec:BdM_single}. Let $G$ denote \label{defG} the subgroup of $\diff(X)$
consisting of those $\phi$ such that 
$\delta\circ\phi\circ\delta^{-1}\colon\partial X\times[0,1/2)\to\partial X\times[0,1)$ 
is of the form $(x^\prime,x_n)\mapsto (\varphi(x^\prime),x_n)$ for some diffeomorphism 
$\varphi\colon\partial X\to \partial X$. 
We are going to use two properties that each $\phi\in G$  satisfies: 
the boundary defining function is preserved ($x_n\circ\phi=x_n$ for
$0\le x_n\le 1/2$),  
and the canonical map $2\phi\colon 2X\to 2X$, defined by 
$2\phi\circ i_\pm=i_\pm\circ\phi,$ 
where $i_\pm\colon X^{\pm}\to 2X$ are the two canonical embeddings of $X$ in $2X$, 
is a diffeomorphism of $2X$. 

Throughout this paper, $\pi\colon Z\to Y$ will denote a fiber bundle
over the compact Hausdorff space $Y$ with fiber $X$ and structure group $G$.
Note, however, that this choice of structure group 
 is just for convenience and can always be (essentially uniquely) arranged
for a general bundle with typical fiber $X$, see the Appendix
\ref{sec:structgroup} for details. 

We denote $Z_y:=\pi^{-1}(y)$. 
Each $Z_y$ is a smooth  manifold with boundary, non-canonically 
diffeomorphic to $X$.
The restriction of $\pi$ to $\partial Z=\cup_y\partial Z_y$ 
is a fiber bundle $\pi_\partial\colon\partial Z\to Y$ with fiber $\partial X$ and 
structure group $\diff(\partial X)$. 

Next we define a bundle of Hilbert spaces, and later a C$^*$-algebra which
will act on its space of sections. This is a bit delicate, as it depends on
some further choices; therefore we give the details. We choose a continuous
family of Riemannian metrics
$(g_y)_{y\in Y}$ with corresponding measures $\mu_y$ on $Z_y$
  and define $H_y:= L^2(Z_y,\mu_y)$. Recall that such a family $(g_y)$ exists:
we can patch them together using trivializations of the bundle and a partition
of unity on $Y$, as the space of Riemannian metrics on $X$ is convex.

The union $\mathfrak{H}=\bigcup_{y\in Y}H_y$
is a fiber bundle of topological vector spaces over $Y$, canonically
associated to $\pi\colon Z\to Y$, with trivializations induced from the
trivializations of $\pi$ in the obvious way. The structure group is the group
of invertible bounded operators on $H$, \emph{equipped with the
  strong topology}.

\begin{remark}\rm 
That we obtain here the strong topology and not the norm topology comes from
the fact that the changes of trivialization are implemented by pullback with
the diffeomorphisms of $G$, and this is continuous in the strong, but not
the norm topology. This makes our considerations about bundles of operators
later quite cumbersome and requires to use the fact that we deal with
pseudodifferential operators.
\end{remark}

Moreover, the choice $(g_y)_{y\in Y}$ 
gives rise to a continuous family of inner products on $\mathfrak{H}$
inducing the given topology of the fibers $H_y$.

Let $\A_y$ be the Boutet de Monvel algebra of order and class zero on
$L^2(Z_y)$. We want to define the bundle of Boutet de Monvel algebras 
$\aleph=\bigcup_{y\in Y}\A_y$ as locally trivial bundle with structure group
the automorphism group of the C$^*$-algebra $\A$ with the
\emph{norm topology}, associated to $Z\to Y$. 

To achieve this, we need the diffeomorphism invariance of
  the Boutet de Monvel algebra in a precise form.

\begin{definition}
  Given $\phi\in G$, let $T_\phi$ 
denote the bounded operator on $L^2(X)$ defined by $f\mapsto f\circ\phi^{-1}$.
\end{definition}

\begin{proposition}\label{prop:conj_is_cont}
  We have a well defined continuous action (for
  the   Fr\'echet topology on $G$ and the norm topology on $\A$)
  \begin{equation*}
    G\times \A\ni (\phi,A) \mapsto T_\phi A T^{-1}_\phi \in\A. 
  \end{equation*}
  Moreover, by restriction we get an action $G\times \A^\circ\to \A^\circ$. 
\end{proposition}
\begin{proof} This corresponds to  \cite[Proposition 1.3]{AS}.  
In fact, even if $X$ is closed, 
Atiyah and Singer consider a slightly different situation in that 
they close $\A^\circ$  with respect to the operator norm of the action
on all Sobolev spaces, while we only use the operator norm on $L^2$.
Their argument still applies verbatim, since they treat the action 
on each Sobolev space separately.  

Indeed, the proof of \cite[Proposition 1.3]{AS} uses only a number of formal 
properties of the algebra of pseudodifferential operators which are also 
satisfied by the Boutet de
Monvel algebra, and therefore applies in the same way to our general
  situation. To be more specific, let us list these properties:
  \begin{enumerate}
  \item the Boutet de Monvel algebra $A^\circ$ is diffeomorphism invariant,
    i.e.~in particular $T_\phi A T^{-1}_\phi\in \A^\circ$ for $A\in \A^\circ$ and
    $\phi\in G$. 
  \item Each $T_\phi$ is a bounded operator on $L^2(X)$ and the map $G\to
    \mathcal{L}(L^2(X))$ is strongly continuous. Moreover, for a sufficiently
    small open neighborhood of $1$, the image has uniformly bounded norm. The
    proof of this fact as given in \cite{AS} works for compact manifolds with
    boundary exactly the same way as for closed manifolds.
  \item Let $\mathcal{V}_G$ denote the space of vector fields on $X$ which, in the
    collar, pull back from vector fields on $\partial X$. 
    The exponential map, defined with the help of Riemannian metrics which
    respect the collar structure, gives a local diffeomorphism (of Fr\'echet
    manifolds)    between  $\mathcal{V}_G$ and $G$.
  \item If $V\in \mathcal{V}_G$ and $A\in \A^\circ$ then the commutator $[A,V]$ 
  belongs to  $ \A^\circ$ by the rules of the calculus, cf. \cite[Theorem 2.7.6]{G}. 
  \end{enumerate}
  All these properties are either well known or easy to establish.
\end{proof}

\begin{corollary}\label{corol:bundle_of_C_algebras}
  We obtain the bundle $\aleph=\bigcup_{y\in Y}\A_y$ of topological algebras
  with bundle of subalgebras $\aleph^\circ=\bigcup_{y\in Y}\A^\circ_y$,
  modelled on $(\A,\A^\circ)$ with structure group the automorphism group of
  $\A$ with its norm topology and the automorphism group of $\A^\circ$ its
  Fr\'echet topology. The local trivializations are induced by the local
  trivializations of $\pi\colon Z\to Y$, where a diffeomorphisms
  $\alpha_y\colon Z_y\to  
  X$ obtained from the trivialization map $\A_y$ to $\A$ by conjugation with
  $T_{\alpha_y}$.

  Moreover, the choice of metrics $(g_y)_{y\in Y}$ induces a continuous family
  of norms 
  on the fibers of $\aleph$ inducing the topology. With these norms the bundle
  becomes a bundle of C$^*$-algebras.
\end{corollary}
\begin{proof}
  The statement about the bundle of topological algebras follows immediately
  from Proposition \ref{prop:conj_is_cont}. Moreover, it is well known that
  each $\A_y$ is closed under taking adjoints in $\mathcal{L}(L^2(Z_y))$.

We now check that with this structure, 
we obtain a locally trivial bundle of
C$^*$-algebras. Fix a local trivialization with diffeomorphisms
$\alpha_y\colon Z_y\to X$. If we pull back  the inner
  products on $H_y$  to $H=L^2(X)$ with the induced maps, 
  then the corresponding Gram operator
  $G_y$, expressing this pullback inner product in terms of the original one
  on $L^2(X)$, is the multiplication with a smooth positive function $m_y$
  which   depends continuously on $y$: the
  density of $\alpha_y^*\mu_y$ with respect to a chosen measure $\mu$ 
  on $X$. Note that $G_y$ belongs
  to $\A$ and its norm, which is just the supremum,
  depends continuously on $y$. 
 Now compose the original trivialization of
  $\A_y$ with conjugation by $\sqrt{G_y}$ and the resulting trivialization
  will respect the C$^*$-algebra structures, but inherit the norm continuity
  of transition maps.  To summarize: with a canonical modification (given in
  terms of the inner products) we have obtained trivializations of our bundle
  $\aleph$ as a bundle of C$^*$-algebras, as claimed.
\end{proof}

\begin{definition}\label{def:Cstar_alg_of_sect}
 We denote by $\ac$ the set of continuous sections of the bundle $\aleph$ of
 C$^*$-algebras. With the pointwise operations and the supremum norm, this
 becomes a C$^*$-algebra. The underlying topological algebra is canonically
 associated to $\pi\colon Z\to Y$, the norm and the $*$-operation depend on
 the choice of the family of metrics $(g_y)_{y\in Y}$. 
\end{definition}

The principal 
symbol and the boundary principal symbol extend continuously to two families 
of C$^*$-algebra homomorphisms
\[
\sigma_y\colon \A_y\rightarrow C(S^*Z_y)
\ \ \mbox{and}\ \ 
\gamma_y\colon \A_y\rightarrow C(S^*\partial Z_y,\mathcal{L}(L^2(\R_{\ge0}))),
\]
where $S^*$ denotes cosphere bundle and $\mathcal{L}$ bounded operators.
Here $\gamma_y$  is well defined, 
since the structure group of the bundle $\pi\colon Z\to Y$
leaves the boundary defining function invariant, see \cite[Theorem 2.4.11]{G}. 

Let us denote by $S^*Z$ the disjoint union of all $S^*Z_y$. This can 
canonically be viewed as the total space of a fiber bundle over $Y$ with 
structure group $G$. 
One analogously defines $S^*\partial Z=\cup_{y}S^*\partial Z_y$ and 
$S^*Z^\circ=\cup S^*Z_y^\circ$.

\begin{df}
Given $A\in\ac$, let $\sigma_A$ be the function on $S^*Z$ defined by 
piecing together all the $\sigma_y$'s. Then $A\mapsto\sigma_A$ defines
a C$^*$-algebra homomorphism 
\[
\sigma\colon \ac\longrightarrow C(S^*Z).
\]
One also gets, analogously,  
\[
\gamma\colon \ac\longrightarrow C(S^*\partial Z,\mathcal{L}(L^2(\R_{\ge0}))).
\]
\end{df}

Let $\kc$ denote the subalgebra of $\ac$ consisting of the sections 
$(A_y)_{y\in Y}$ such that $A_y$ is compact for every $y\in Y$. It follows
immediately from the corresponding statement for a single manifold that
$
\ker\sigma\cap\ker\gamma=\kc.
$
It is also straightforward to generalize the description of $\ker\gamma$ 
for a single manifold \eqref{eq:ker_gamma_single_mf}:

\begin{thm} \label{ker}
The principal symbol restricted to $\ker\gamma$ induces a C$^*$-algebra
isomorphism
\begin{equation}
\label{kernel}
\ker\gamma/\kc\simeq C_0(S^*Z^\circ).
\end{equation}
Here $C_0(S^*Z^\circ)$ consists of the elements of $C(S^*Z)$ which, 
for every $y\in Y$, vanish on all points of $S^*Z_y$ with base point belonging to 
$\partial Z_y$.
\end{thm}

Regarding each $f\in C(Z)$ as a family of multiplication operators on 
$(H_y)_{y\in Y}$, 
furnishes an embedding of $C(Z)$ in $\ac$, which we denote 
$m\colon C(Z)\to\ac$. 
Mapping a $g\in C(\partial Z)$ to 
the boundary principal symbol of $m(f)$, where $f\in C(Z)$ is such that its  
restriction to $\partial Z$ is $g$, defines the C$^*$-algebra homomorphism 
$b\colon C(\partial Z)\to\im\gamma$. 

\begin{thm}\label{biso} The homomorphisms 
$b_*\colon K_i(C(\partial Z))\to K_i(\im\gamma)$, 
$i=0,1$, induced by $b$ are isomorphisms. 
\end{thm}
\pf
Given an open set $U\subseteq Y$, let us denote by 
$\pi_U\colon Z_U=\pi^{-1}(U)\to U$ the 
restriction of $\pi$ to $U$, by $\ac_U$ the algebra of sections in $\ac$ which vanish 
outside $U$ and by $\gamma_U$ the restriction of $\gamma$ to $\ac_U$.
Moreover we let 
$$C_0(\partial Z_U)=\{f\in C(\partial Z)\colon  \text{supp}\, f\subseteq
\pi_\partial^{-1}(U)\}
$$
and write $b_U$ for the restriction of $b$ to $C_0(\partial Z_U)$.
If the bundle $\pi$
is trivial over $U$, then $\ac_U$ is isomorphic to $C_0(U,\A)$ and, with respect to
this isomorphism, $b_U$ corresponds to the tensor product of the identity on $C_0(U)$
with the corresponding map for a single manifold, also denoted by $b$ on \cite{MNS,MSS}.
It is the content of \cite[Corollary 8]{MNS} that $b$ induces a K-theory isomorphism 
onto the image of $\gamma$.
It then follows from the K\"unneth formula for C$^*$-algebras \cite{S}
that $b_U$ induces isomorphisms 
$
b_{U*}\colon K_i(C_0(\partial Z_U))\longrightarrow K_i(\im\gamma_U)
$, $i=0,1$, see Proposition \ref{kun} in Appendix \ref{sec:kunneth}.

Now let $(\im\gamma)_U$ denote the subset of $\im\gamma$ consisting 
of those functions which vanish outside 
${\displaystyle \cup_{y\in U}S^*\partial Z_y}$.
It is obvious that 
$\im \gamma_U\subseteq (\im \gamma)_U$. Since both 
$\im \gamma_U$ and $(\im \gamma)_U$ are closed in 
$C(S^*\partial Z,\mathcal{L}(L^2(\R_{\ge0})))$, to show that they are equal it 
suffices to show that the former is dense in the latter. This follows from 
the fact that multiplication by a complex continuous function with support
contained in $U$ maps $(\im \gamma)_U$ to $\im \gamma_U$. 
This simple observation implies that, for open sets $U$ and $V$,  
we have a canonical C$^*$-algebra isomorphism
\begin{equation}\label{hannover}
\im \gamma_{U\cap V}\cong
\{(f,g)\in\im \gamma_U\oplus\im \gamma_V;f=g \}.
\end{equation}

Now suppose that we have shown $b_{U*}$ to be an isomorphism
for some open $U$ and that $V$ is open and $\pi$ trivial over
  $V$, and so in particular also over $U\cap V$.  
We then consider the two --- thanks to \eqref{hannover} ---
diagrams  
\[
\begin{array}{ccc}
C_0(\partial Z_{U\cap V})&\to&C_0(\partial Z_U)\\
\downarrow&&\downarrow\\
C_0(\partial Z_V)&\to&C_0(\partial Z_{U\cup V})
\end{array}
\ \ \mbox{and}\ \ 
\begin{array}{ccc}
\im \gamma_{U\cap V}&\to&\im \gamma_U\\
\downarrow&&\downarrow\\
\im \gamma_V&\to&\im \gamma_{U\cup V}
\end{array}.
\]
Because they are cartesian, we may extract from both diagrams
  cyclic exact Mayer-Vietoris sequences  (see \cite[21.2.2]{Bl} or 
\cite[7.2.1]{MS}), and we may use the K-theory maps induced
by $b_U$, $b_V$, $b_{U\cap V}$ and $b_{U\cup V}$ to map the first cyclic sequence to the
second. By assumption and the case of trivial bundles, the
maps induced by $b_U$, $b_V$ and $b_{U\cap V}$ are isomorphisms. It then
follows from the five-lemma that also $b_{U\cup V}$ induces a K-theory
isomorphism. 

Since $Y$ has a finite cover by open sets over which $\pi$ is trivial, 
induction shows that $b$ induces K-theory isomorphisms.
\cqd

Using Theorem~\ref{ker}, we obtain the following commutative diagram of 
C$^*$-algebra homomorphisms, whose horizontal lines are exact:
\[
\def\mapup#1{\Big\uparrow\rlap{$\vcenter{\hbox{$\scriptstyle#1$}}$}}
\begin{array}{ccccccc}
0\longrightarrow&C_0(S^*Z^\circ)&\longrightarrow&\mathfrak{A}/\mathfrak{K}&{\mathop{\longrightarrow}\limits^\gamma}&\text{Im}\gamma&\longrightarrow 0
\\                &\mapup{m^\circ}&&\mapup{m}&&\mapup{b}&
\\0\longrightarrow&C_0(Z^\circ)&\longrightarrow&C(Z)&{\mathop{\longrightarrow}\limits^r}&C(\partial Z)&\longrightarrow 0
\end{array}.
\]
We have denoted by $r$ the map that pieces together all restrictions 
$r_y\colon C(Z_y)\to C(\partial Z_y)$, $y\in Y$, and by $Z^\circ$ the union
$\cup_yZ_y^\circ$. Since the isomorphism (\ref{kernel}) is induced by 
the principal symbol, and the principal symbol of an operator of multiplication by
a function is the function itself, the map $m^\circ$ in the diagram above is 
actually the map of composition with the canonical projection $S^*Z^\circ\to Z^\circ$. 
We may apply the cone-mapping functor \cite[Lemma 9]{MSS} to the above diagram 
and get (using the same arguments that prove (11) in \cite{MSS}) the following 
commutative diagram of cyclic exact sequences
\begin{equation}
\label{2ces}
\def\mapdown#1{\downarrow\rlap{$\vcenter{\hbox{$\scriptstyle#1$}}$}}
\def\mapup#1{\uparrow\rlap{$\vcenter{\hbox{$\scriptstyle#1$}}$}}
\begin{array}{ccc}
K_0(C_0(Z^\circ))          &\longrightarrow                         &K_0(C(Z))\\
\mapdown{m^\circ_*}&                                        &\mapdown{m_*}\\
K_0(C_0(S^*Z^\circ))     &\longrightarrow &K_0(\ac/\kc)\\
{\downarrow}          &                                        &{\downarrow}\\
K_1(Cm^\circ)              &{\mathop{\longrightarrow}\limits^{\cong}}&K_1(Cm)\\
\downarrow               &                                        &\downarrow\\
K_1(C_0(Z^\circ))          &\longrightarrow                         &K_1(C(Z))\\
\mapdown{m^\circ_*}&                                        &\mapdown{m^\circ_*}\\
K_1(C_0(S^*Z^\circ))     &\longrightarrow &K_1(\mathfrak{A}/\mathfrak{K})\\
 \downarrow          &                                        &\downarrow\\
K_0(Cm^\circ)              &{\mathop{\longrightarrow}\limits^{\cong}}&K_0(Cm)\\
\downarrow               &                                        &\downarrow\\
K_0(C_0(Z^\circ))          &\longrightarrow                         &K_0(C(Z))
\end{array},
\end{equation}
where $\cong$ denotes isomorphism.

Up to this point, everything goes exactly as in the case of a single manifold,
but here comes a difference: 
The homomorphism $m_0$ does not necessarily 
have a left inverse (in the case of a single manifold $X$,
such a left inverse is defined 
by composition with a section of $S^*X$), and hence the cyclic exact sequences above
do not have to split into short exact ones. 

To proceed we now introduce the subalgebra $\mathfrak A^\dagger$ 
of $\mathfrak A$ and an associated subalgebra $B$ of $C(S^*Z)$
with the properties outlined in the introduction:
For each $y\in Y$, let $B_y$ denote the subalgebra of $C(S^*Z_y)$ consisting of 
the functions which do not depend on the co-variable over the boundary, that 
is, an $f\in C(S^*Z_y)$ belongs to $B_y$ if and only if the restriction of $f$ 
to the points of $S^*Z_y$ over $\partial Z_y$ equals $g\circ p_y$, for some
$g\in C(\partial Z_y)$, where $p_y\colon S^*Z_y\to Z_y$ is the canonical projection. 
We then define $\A_y^\dagger$ as the C$^*$-subalgebra of $\A_y$ generated by  
$\{P_+;\ P$ is a pseudodifferential operator with the 
transmission property and $\sigma_y(P_+)\in B_y\}$. 

\begin{df}Let $B$ denote the subalgebra of $C(S^*Z)$ consisting of the
functions whose restriction to each $S^*Z_y$ belongs to $B_y$.
We let then $\ac^\dagger$ be the C$^*$-subalgebra of $\ac$ consisting of the 
sections $(A_y)_{y\in Y}$ such that $A_y\in\A_y^\dagger$ for every $y\in Y$.
\end{df}

\begin{pro}\label{hoc} 
The C$^*$-algebra
$\ac^\dagger/\kc$ is commutative, and the map
\[
\ac^\dagger/\kc\ni[A]{\mathop{\longmapsto}\limits^{\bar\sigma}}\sigma(A)\in B
\]
is a C$^*$-algebra isomorphism. 
\end{pro}
\begin{proof}
    Let $P=(P_y)$ be a family of \emph{pseudodifferential} operators with
    symbol independent of the co-variable over the boundary, 
     i.e.~a generator
    of $\ac^\dagger$. According to 
\eqref{sec:def_of_boundary_symb}, $\gamma(P)$ can 
be considered as a
    function on $\partial Z$, acting for $z\in\partial Z$ on $L^2(\reals_{\ge
      0})$ by multiplication with $\gamma(P)(z)$. Moreover, for $z\in\partial
    Z$ we have $\gamma(z)=\sigma(z)$ independent of the
    co-variable by assumption. 
    It follows that the composed algebra homomorphism
    \begin{equation*}
      \sigma\colon \ac^\dagger\xrightarrow{\sigma\oplus\gamma} 
C(S^*Z)\oplus
      C(S^*\partial Z, \mathcal{L}(L^2(\R_{\ge0}))) \xrightarrow{pr} C(S^*Z)
    \end{equation*}
    has the same kernel as $\sigma\oplus\gamma$, namely $\kc$ and so the map
    we consider is injective and in particular $\ac^\dagger/\kc$ is
    commutative. By the very definition of $\ac^\dagger$,
    $\sigma\colon\ac^\dagger\to B$ has dense image, as a morphism of
    C$^*$-algebras it is therefore also surjective.  
\end{proof}

This allows us to describe the K-theory of $\ac/\kc$: 
\begin{thm}\label{Kth}
The composition
\[
K_i(\ac/\kc){\mathop{\longrightarrow}\limits^{\iota_*^{-1}}}K_i(\ac^\dagger/\kc)
{\mathop{\longrightarrow}\limits^{\bar\sigma_*}}K_i(B)
\]
is an isomorphism, $i=0,1$.
\end{thm}

The proof makes use of the following proposition, which is easily established
by a diagram chase, compare \cite[Exercise 38 of Section 2.2]{Hatcher}: 

\begin{pro}\label{ditsche} Let there be given a commutative diagram of 
abelian groups with exact rows, 
\[
\def\mapdown#1{\downarrow\rlap{$\vcenter{\hbox{$\scriptstyle#1$}}$}}
\def\mapup#1{\uparrow\rlap{$\vcenter{\hbox{$\scriptstyle#1$}}$}}
\begin{array}{ccccccccccc}
\cdots&\rightarrow
&A_i^\prime&{\mathop{\longrightarrow}\limits^{f_i^\prime}}
&B_i^\prime&{\mathop{\longrightarrow}\limits^{g_i^\prime}}
&C_i^\prime&{\mathop{\longrightarrow}\limits^{h_i^\prime}}
&A_{i+1}^\prime
&\rightarrow&\cdots
\\
&
&\mapup{a_i}&&\mapup{b_i}&&\mapup{c_i}&&\mapup{a_{i+1}}&
&
\\
\cdots&\rightarrow
&A_i&{\mathop{\longrightarrow}\limits^{f_i}}
&B_i&{\mathop{\longrightarrow}\limits^{g_i}}
&C_i&{\mathop{\longrightarrow}\limits^{h_i}}
&A_{i+1}
&\rightarrow&\cdots
\end{array},
\]
where each $c_i$ is an isomorphism. Then the sequence
\[
\cdots\longrightarrow A_i{\mathop{\longrightarrow}\limits^{(a_i,-f_i)}}
A_i^\prime\oplus B_i{\mathop{\longrightarrow}\limits^{\langle f_i^\prime,b_i\rangle}}
B_i^\prime{\mathop{\longrightarrow}\limits^{h_ic_i^{-1}g_i^\prime}}
A_{i+1}\longrightarrow\cdots
\]
is exact, where $\langle f_i^\prime,b_i\rangle$ is the map defined by
$\langle f_i^\prime,b_i\rangle(\alpha,\beta)=f_i^\prime(\alpha)+b_i(\beta)$.
\end{pro}

We are now ready to prove Theorem \ref{Kth}.
Applying Proposition \ref{ditsche} to the diagram (\ref{2ces}), we get the exact sequence
\begin{equation}
\label{cyclic1}
\begin{array}{ccccc}
K_0(C_0(Z^\circ))&\rightarrow &K_0(C(Z))\oplus K_0(C_0(S^*Z^\circ))&\rightarrow& 
K_0(\ac/\kc)
\\\uparrow& & & &\downarrow
\\
K_1(\ac/\kc)&\leftarrow& K_1(C(Z))\oplus K_1(C_0(S^*Z^\circ))&\leftarrow&  
K_1(C_0(Z^\circ))
\end{array}.
\end{equation}

We next consider the following diagram of commutative C$^*$-algebras
\begin{equation}\label{cd3}
\def\mapdown#1{\downarrow\rlap{$\vcenter{\hbox{$\scriptstyle#1$}}$}}
\def\mapup#1{\uparrow\rlap{$\vcenter{\hbox{$\scriptstyle#1$}}$}}
\begin{array}{ccc}
C_0(Z^\circ)&{\mathop{\longrightarrow}\limits^{m^\circ}}&C_0(S^*Z^\circ)\\
\downarrow&&\mapdown{p_2}\\
C(Z)&{\mathop{\longrightarrow}\limits^{p_1}}&B
\end{array}.
\end{equation}
As $C_0(Z^\circ)$ is canonically isomorphic to
\[
\{(f,g)\in C(Z)\oplus C_0(S^*Z^\circ);\ p_1(f)=p_2(g)\},
\]
the Mayer-Vietoris exact sequence  associated to (\ref{cd3}) is the 
exact sequence
\begin{equation}
\label{cyclic2}
\begin{array}{ccccc}
K_0(C_0(Z^\circ))&\rightarrow &K_0(C(Z))\oplus K_0(C_0(S^*Z^\circ))&\rightarrow& 
K_0(B)
\\\uparrow& & & &\downarrow
\\
K_1(B)&\leftarrow& K_1(C(Z))\oplus K_1(C_0(S^*Z^\circ))&\leftarrow&  
K_1(C_0(Z^\circ))
\end{array}.
\end{equation}
The map $\iota\colon B\cong \ac^\dagger/\kc\hookrightarrow \ac/\kc $ and the identity on 
the other K-theory groups furnish morphisms from the 
cyclic sequence (\ref{cyclic2}) to  the cyclic sequence (\ref{cyclic1}).
The five  lemma then shows that the induced maps in K-theory are
isomorphisms.  Together with Proposition \ref{hoc} we obtain the assertion. 
\cqd

%

\section{The Boutet de Monvel family index theorem}\label{index}

The index of a continuous function with values in Fredholm operators was defined by 
J\"anich \cite{J} and Atiyah \cite{A}. 
Using the following Proposition \ref{perturb},
  their definition can be extended to sections of our $\aleph$.

\begin{pro}\label{perturb} Let $\hc$ and $\ac$ be as above, $k\in \naturals$  and let 
$(A_y)_{y\in Y}\in M_k(\ac)$ be such that, for each $y$, 
$A_y$ is a Fredholm operator, where we interpret $M_k(\ac)$ as the sections of
the bundle with fiber $M_k(\A_y)$. Then there are continuous sections
$s_1,\cdots,s_q$ of $\mathfrak{H}^k$ such that the maps
\[
\begin{array}{rccl}
\tilde A_y\colon &H_y^k\oplus\C^q&\longrightarrow&H_y^k\oplus\C^q\\
              &(v,\lambda)&\longmapsto&(A_yv+\sum_{j=1}^{q}\lambda_js_j(y),0)
\end{array}
\]
have image equal to $H_y^k\oplus 0$ for all $y\in Y$ and 
$(\ker\tilde A_y)_{y\in Y}$ is a (finite dimensional) vector bundle 
over $Y$. 
\end{pro}
\pf
Similar to \cite[Proposition (2.2)]{AS} and to \cite[Proposition A5]{A}.
\cqd

\begin{df}\label{ani} 
Given $A=(A_y)_{y\in Y}\in\ac$ as in Proposition~\ref{perturb}, we
denote by $\ker\tilde A$ the bundle $(\ker\tilde A_y)_{y\in Y}$ and
define
\[
\ind_a(A)=[\ker\tilde A]-[Y\times\C^q]\in K(Y).
\]
This is independent of the choices of $q$ and of $s_1,\cdots,s_q$ and 
we call it {\em the analytical index of} $A$. 
\end{df}

If $A=(A_y)_{y\in Y}\in M_k(\ac)$ 
is a section such that each $A_y$ is a Fredholm operator on $H_y^k$ then 
the projection to $M_k(\ac/\kc)$ 
is invertible and hence defines an element of $K_1(\ac/\kc)$. Since $\ind_a(A)$
is invariant under stabilization, homotopies and perturbations 
by  
compact operator valued sections, we get a homomorphism 
\begin{equation}
\label{topind}
\ind_a\colon  K_1(\ac/\kc)\longrightarrow K(Y).
\end{equation}

$\ $

Next we define the \emph{topological} index, also 
as a homomorphism 
\[
\ind_t\colon  K_1(\ac/\kc)\longrightarrow K(Y).
\]

Let $T^*Z$ denote the union 
of all $T^*Z_y$, and $B^*Z$ the union of all $B^*Z_y$, equipped with their 
canonical topologies, where $B^*Z_y$ denotes the bundle of closed unit 
balls of $T^*Z_y$. One may regard $B^*Z$ as a compactification of $T^*Z$ and 
identify the ``points at infinity'' with $S^*Z$.

Let $\sim$ denote the equivalence relation that identifies, for each 
$y\in Y$, all points of each ball of $B^*Z_y$ which lies over a point of 
$\partial Z_y$. The C$^*$-algebra $B$ of Theorem~\ref{Kth} is isomorphic 
to the algebra of continuous functions on the quotient space 
$S^*Z/\asim$. Let 
$\beta\colon K_1(C(S^*Z/\asim))\to K_0(C_0(T^*Z^\circ))$ denote 
the index map associated to the short exact sequence
\[
0\longrightarrow C_0(T^*Z^\circ)\longrightarrow C(B^*Z/\asim)
\longrightarrow C(S^*Z/\asim)\longrightarrow 0,
\]
where $T^*Z^\circ$ is the union over $y\in Y$ of all points of $T^*Z_y$
which lie over interior points of $Z_y$ and the map from  
$C(B^*Z/\asim )$ to $C(S^*Z/\asim )$ is induced by restriction.

Let $2Z$ denote the union $\cup_y2Z_y$, where each $2Z_y$ is the
double of $Z_y$, and $\pi_d\colon 2Z\to Y$ the canonical projection. 
This can be given the structure of a $\diff(2X)$-bundle, with trivializations 
obtained by ``doubling'' (as explained at the beginning of Section \ref{sec1})
the trivializations 
of the bundle $\pi\colon Z\to Y$. Each fiber $2Z_y$ is then equipped with the 
smooth structure induced by the trivializations of
$\pi_d\colon 2Z\to Y$ and 
we can form the bundles  $T^*2Z$ and $S^*2Z$ as 
\label{coesferas} the unions, respectively, of all 
cotangent bundles $T^*(2Z_y)$ and of all cosphere bundles $S^*(2Z_y)$, $y\in Y$. 
We denote by $\mbox{{\sc as}-ind}_t\colon K_0(C_0(T^*2Z))\to K(Y)$ the composition
of Atiyah and Singer's \cite{AS}\ topological families-index for 
the bundle of closed manifolds $2Z$  with the canonical 
isomorphism $K(T^*2Z)\simeq K_0(C_0(T^*2Z))$.
Theorem~\ref{Kth} allows us to define the topological index:
\begin{df}\label{indt} 
The topological index $\ind_t$ is the following composition of maps
\begin{equation*}
\label{anaind}
\def\mapdown#1{\downarrow\rlap{$\vcenter{\hbox{$\scriptstyle#1$}}$}}
\def\mapup#1{\uparrow\rlap{$\vcenter{\hbox{$\scriptstyle#1$}}$}}
\begin{array}{rc}
\ind_t\colon K_1(\ac/\kc)
{\mathop{\longrightarrow}\limits^{\bar\sigma_*\circ\iota_{*}^{-1}}}
K_1(C(S^*Z/\asim )){\mathop{\longrightarrow}\limits^{\beta}} 
K_0(C_0(T^*Z^\circ)){\mathop{\longrightarrow}\limits^{e_*}}
&K_0(C_0(T^*2Z))\\ &
\mapdown{\mbox{{\sc as}}-\ind_t}\\ &K(Y),
\end{array}
\end{equation*}
where $e\colon C_0(T^*Z^\circ)\to C_0(T^*2Z)$ denotes the map
which extends by zero.
\end{df}

If $A=(A_y)_{y\in Y}\in\ac$ is a family of Fredholm operators
we denote by $\ind_t(A)$ the topological index evaluated
at the element of $K_1(\ac/\kc)$ that $A$ defines.

\begin{thm}\label{indthm} 
Let $A=(A_y)_{y\in Y}\in\ac$ be a continuous 
family of Fredholm operators in the closure of the Boutet de Monvel algebra for 
each $y$. Then{\em 
\begin{equation}
\label{teoind}
\ind_a(A)\ =\ \ind_t(A).
\end{equation} }
\end{thm}

\pf 
Our strategy is to derive the equality of the indices from the classical 
Atiyah-Singer index theorem for families \cite[Theorem (3.1)]{AS}. 
To this end we define an operator family $\hat A$ acting on a vector bundle over 
the double of $Z$ by a gluing technique involving the principal symbol family of $A$. 
We proceed in several steps. 
Step 1 consists of a few preliminary remarks on the choice of the representative 
of the K-theory class of $A$. 
In Step 2 we describe the construction of the bundle. 
We then define the operator family $\hat A$ over $2Z$ in Step 3.  
Its topological index coincides with that of $A$ as we shall see in Step 4. 
The equality of the analytic indices of $A$ and $\hat A$ is
the content of Step 5.  

{\em Step} 1. 
We need to prove that $\ind_t$ and $\ind_a$ 
coincide on $K_1(\ac/\kc)$.
Using that $K_1(\ac/\kc)=K_1(\ac^\dagger/\kc)$ by Theorem \ref{Kth}, an
arbitrary element of $K_1(\ac/\kc)$ is a class  
$[[A]]_1$ (the inner brackets denoting a class in the quotient by the
compacts), for some operator family
$A=(A_y)_{y\in Y}\in M_k(\ac^\dagger)$,  $k\in \mathbb N$,  
such that, for each $y$,  $A_y\colon H_y^k\to H_y^k$ is a Fredholm operator
with symbol in $B$.
It will be convenient to pick a representative with special properties. 
We denote by $C^\infty(S^*X/\asim)$  the subset of $C^\infty(S^*X)$
of functions which factor through $S^*X/\asim$, i.e.\ are independent 
of the co-variable at the boundary.
The algebraic tensor product 
$C_0(U)\otimes C^\infty(S^*X/\asim)$
is dense in $C(U\times S^*X/\asim)$ for every open subset $U$ of $Y$. 
Furthermore, the inclusion of the space of all elements in $C^\infty(S^*X/\asim)$  
which are independent of the co-variable even in a neighborhood of $\partial Z$
into $C^\infty(S^*X/\asim)$ is a homotopy equivalence.
We can therefore assume that the symbol family
$(\sigma_y(A_y))_{y\in Y}$ is given as a finite sum of elements
supported in  open subsets $U$  of $Y$ over which $Z$ is trivial, 
and each of these is a  pure tensor in 
$C_0(U)\otimes C^\infty(S^*X)$
which is independent of the co-variable near the boundary.
Hence it suffices to prove equality for such an $A$. 

{\em Step} 2. 
For each $y\in Y$, let $Z_y^+$ and $Z_y^-$ denote the two copies of $Z_y$ which 
are glued together at $\partial Z_y$ to form $2Z_y$. The map
$i_y\colon \partial Z_y^+\to\partial Z_y^-$ identifies the two copies of 
$\partial Z_y$. We define $E_y$ as the quotient of the disjoint union 
$Z_y^+\times\C^k\cup Z_y^-\times\C^k$ by the equivalence relation that 
identifies the pairs $(x,v)$ and $(x^\prime,w)$ if and only if they are 
equal or $x^\prime=i_y(x)$, $x\in\partial Z_y^+$, and 
$w=\sigma_y(A_y)(x) v$ (remembering that at points of $S^*Z_y$ over 
$\partial Z_y$, $\sigma_y(A_y)$ is independent of the co-vector variable). 
This set $E_y$ naturally becomes a smooth vector 
bundle over $Z_y$. Let $E$ denote the union 
of all $E_y$, which in the same way becomes a vector bundle over $Y$. 

When defining families of smooth manifolds with smooth vector bundles, 
Atiyah and Singer make the technical assumption that the fiberwise vector bundles
are isomorphic to a fixed vector bundle on the typical fiber.
If $Y$ is not connected, this is not necessarily satisfied. 
However, the isomorphism type of $E_y$ depends only on the homotopy 
type of the map $\sigma_y$, in particular only on the component of the space
of all continuous maps from $\partial Z_y$ to $M_k(\C)$ in which it lies. By
the compactness of $Y$, the latter decomposes into finitely many open and
closed subsets over each of which the isomorphism type of $E_y$ is
constant. As the K-theory of $Y$ as well as $\ac/\kc$ split as direct sums
under such disjoint union decompositions of $Y$, and as 
$\ind_a$, $\ind_t$ respect
this, we can restrict to one such subset of $Y$. Then 
we are canonically in the situation of \cite[Definition 1.2]{AS},
i.e.~$E$ is a smooth vector bundle over the family of smooth manifolds $2Z$.

{\em Step} 3.
Let $\pi_s\colon S^*2Z\to 2Z$ denote the canonical projection and  
$S^*Z^+$ and $S^*Z^-$, respectively, the union of all $S^*Z_y^+$ and $S^*Z_y^-$, $y\in Y$. 
The bundle $\pi_s^*E$ can be seen as the disjoint union
of $S^*Z^+\times\C^k$ and $S^*Z^-\times\C^k$ quotiented by the equivalence relation that identifies 
a boundary point $(s,v)$ in $S^*Z^+\times\C^k$ with  
$(s,\sigma_A(s)\cdot v)$ in $S^*Z^-\times\C^k$  .\ 
Similarly, the bundle $S^*2Z\times\C^k$ can be seen as the disjoint union
of $S^*Z^+\times\C^k$ and $S^*Z^-\times\C^k$ quotiented by the equivalence relation that identifies 
a boundary point $(s,v)$ in $S^*Z^+\times\C^k$ with $(s,v)$ in
$S^*Z^-\times\C^k$.
We then define $\hat{a}\in\mbox{Hom}(\pi_s^*E,\,S^*2Z\times\C^k)$ by
\begin{equation}\label{ahat}
\hat{a}(s,v)=\left\{\begin{array}{rl}
\sigma_A(s)\cdot v,\ &\mbox{if}\ (s,v)\in S^*Z^+\times\C^k,\\
v,\ &\mbox{if}\ (s,v)\in S^*Z^-\times\C^k.
\end{array}\right.
\end{equation}
We want to show that $\hat{a}$ is the 
symbol of a continuous family of pseudodifferential operators. 
As any element of $\mbox{Hom}(\pi_s^*E,\,S^*2Z\times\C^k)$, 
our $\hat{a}$ can be regarded as a family $(\hat{a}_y)_{y\in Y}$, $\hat{a}_y\in\mbox{Hom}(\pi_s^*E_y,\,S^*2Z_y\times\C^k)$.
It is easily checked that our 
definition of $\hat{a}$ indeed mends continuously at boundary points. 
But more is true. 
Since $\sigma_y(A_y)$ is smooth and independent of the co-variable 
near the boundary, each $\hat a_y$ is smooth. 
Moreover, since we assumed in Step 1 that 
$a$ is a finite sum of local elementary tensors, 
we see that $\hat{a}$ is the symbol of an Atiyah-Singer family of pseudodifferential operators on $2Z$%
\footnote{Recall that they use a 
slightly stricter definition of operator families: 
While we here require continuity of the family with respect to the 
$L^2(X)$-operator norm, they take into account the norms on the whole range of 
Sobolev spaces.}.

{\em Step} 4. 
Let $\iota \colon K_0(C_0(T^*2Z))\to K(B^*2Z,S^*2Z)\simeq
K(T^*2Z)$ denote the canonical isomorphism 
(we refer to  \cite{Bl} and mainly 
\cite{Karoubi} for topological K-theory
definitions and notation).  
By Definition \ref{indt}, it is enough to show that 
$\iota(e_*(\beta([\sigma_A]_1)))$ is equal to the element of $K(B^*2Z,S^*2Z)$ defined by the triple $(\pi_b^*E,\,B^*2Z\times\C^k, \,\hat{a})$, 
where $\pi_b\colon B^*2Z\to 2Z$ denotes the canonical projection. 

The main step here is to understand $\beta([\sigma_A]_1)$. 
Now, $\sigma_A$ can
and will be considered as a function on $S^*Z/\asim$ with values in
$Gl_k(\complexs)$, representing an element in $K_1(C(S^*Z/\asim))$ and at the
same time the corresponding element of the topological K-theory
$K^1(S^*Z/\asim)$, \cite[3.2]{Karoubi}. Recall from \cite[3.21]{Karoubi}
that for the pair of compact topological spaces
$S^*Z/\asim\ \subset B^*Z/\asim$, 
 the boundary map in topological K-theory assigns to $\sigma_A$ 
the relative K-class 
$((B^*Z/\asim)\times \C^k,(B^*Z/\asim)\times\C^k,\sigma_A)$, 
corresponding under the excision isomorphism $K((B^*Z/\asim),$ $(S^*Z/\asim))\iso
K(B^*Z,S^*Z)$ to $(B^*Z\times \C^k,B^*Z\times \C^k,\sigma_A)$, compare
\cite[2.35]{Karoubi}. Moreover, this corresponds to $\beta$ under the
isomorphism 
with C$^*$-algebra K-theory. 
We next have to compute the map 
$e^{top}\colon K(B^*Z,S^*Z)\to K(B^*2Z,S^*2Z)$ in topological K-theory, representing
$e_*\colon K_0(C_0(T^*Z))\to K_0(C_0(T^*2Z))$. 
Recall, however, that 
$e^{top}(V,W,\tau)$ is given by any extension $\tilde V$ of $V$, $\tilde W$ of $W$
to $B^*2Z$ and an extension of $\tau$ to an isomorphism $\tilde \tau$ between $\tilde V$ and
$\tilde W$ on all of $(B^*2Z\setminus B^*Z)\cup S^*Z$, $\tilde \tau$ finally
restricted to $S^*2Z$. Finally, observe that
$(\pi_b^*E, B^*2Z\times \C^k, \hat{a})$ provides exactly such an extension (as
$\hat{a}$ extends as $id$ over all of $B^*2Z\setminus B^*Z$) and therefore
represents $\iota e_*(\beta([\sigma_A]))$, as we had to prove.

{\em Step} 5. 
In order to show that the analytic indices coincide, we will introduce 
yet another operator family.
Since  $\sigma(A)$ is independent of the co-variable near the boundary, 
there is an open set $U\subseteq 2Z$ containing $Z^-=\cup_yZ^-_y$ 
and a bundle isomorphism 
\[
\Phi\colon  E|_{U}\longrightarrow U\times\C^k
\]
such that the restriction of $\hat{a}$ to $\pi_s^{-1}(U)$ is equal to the pullback of $\Phi$ by $\pi_s$.
Let $(\chi_y^+)_{y\in Y}$ and $(\chi_y^-)_{y\in Y}$ be continuous families of smooth functions on $2Z$ with 
$0\leq\chi_y^\pm\leq 1$, $(\chi_y^+)^2+(\chi_y^-)^2=1$.
Moreover, let the support of each $\chi^+_y$ be contained in the interior of $Z_y^+$ and  $\chi_y^+\equiv 1$ 
outside a neighborhood of $\partial Z_y^+$ in $U$.  
Then
\[
\hat{B}_y=\chi^+_y\hat{A}_y\chi^+_y+\chi^-_y\Phi_y\chi^-_y,
\]
defines a family of pseudodifferential operators in the sense of Atiyah and Singer 
which has the same principal symbol -- and hence the same analytic index --
as $\hat{A}$. 

For each $y\in Y$, we canonically identify the space $L^2(E_y)$ of $L^2$-sections of $E_y$ with the direct sum 
$L^2(Z_y^+;\C^k)\oplus L^2(Z_y^-;\C^k)$ and denote by $e_y^\pm$ and $r_y^\pm$ the maps of extension by zero and restriction,
\[
e_y^\pm\colon L^2(Z_y^\pm;\C^k)\to L^2(E_y)\ \ \ \mbox{and}\ \ \ r_y^\pm\colon L^2(2Z_y;\C^k)\to L^2(Z_y^\pm;\C^k). 
\]
Then  $B_y=r_y^+\hat{B}_ye_y^+$ defines a continuous family 
$B=(B_y)_{y\in Y}$ in $M_k(\ac)$.
As $\sigma(A)=\sigma(B)$ (and hence $\gamma(A)=\gamma(B)$), it suffices to prove that the analytic indices of $B$ and $\hat{B}$ are equal.

Proposition (2.2) of \cite{AS}, applied to the family $\hat{B}$ provides us with sections 
$s_y^j\in C^\infty(2Z_y;\C^k)$, $y\in Y$, $1\leq j\leq q$, such that 
\[
\begin{array}{rcl}
\hat{Q}_y\colon C^\infty(2Z_y;E_ y)\oplus\C^q&\longrightarrow&C^\infty(2Z_y;\C^k)\\
(u;\lambda_1,\cdots,\lambda_q)&\longmapsto&\hat{B}_y(u)+\sum_{j=1}^{q}\lambda_js^j_y
\end{array}
\]
is onto, $\ker \hat{Q}=(\ker \hat{Q}_y)_{y\in Y}$ is a vector bundle and the analytic index of $\hat{B}$ is equal to 
$[\ker \hat{Q}]-[Y\times\C^q]$. 
Now let $t_y^j=r_y^+s_y^j\in C^\infty(Z_y;\C^k)$. 
The continuity with respect to $y$ that we
get from \cite[Proposition (2.2)]{AS} is enough to ensure that $(t_y^j)_{y\in Y}$ is a continuous section of our bundle of Hilbert spaces ${\displaystyle \bigcup_{y\in Y}L^2(Z_y;\C^k)}$. 
We then define
\[
\begin{array}{rcl}
Q_y\colon L^2(Z_y;\C^k)\oplus\C^q&\longrightarrow&L^2(Z_y;\C^k)\\
(u;\lambda_1,\cdots,\lambda_q)&\longmapsto&B_y(u)+\sum_{j=1}^{q}\lambda_jt^j_y
\end{array}
\]
Since $B_y$ is elliptic, $\ker Q_y\subset C^\infty(Z_y;\C^k)$. Using that $\Phi_y$ is local, it is straightforward to check that 
\[
\hat{B}_y= e_y^+r_y^+\hat{B}_ye_y^+r_y^+ +e_y^-r_y^-\hat{B}_ye_y^-r_y^-=e_y^+B_yr_y^++e_y^-r_y^-\Phi_ye_y^-r_y^-
\]
and, hence, $\ker Q_y$ and $\ker\hat{Q}_y$ are isomorphic for each $y$ (because $\Phi$ is an isomorphism). Moreover, $Q_y$ is also surjective: Given $v\in L^2(Z_y;\C^k)$, if $u\in L^2(2Z_y;E_y)$ is a preimage of $e^+_yv$ under $\hat{Q}_y$, then
$r^+_yu$ is a preimage of $v$ under $Q_y$. Hence the analytic index of $B$ is given by $[\ker Q]-[Y\times\C^q]$. The bundles 
$\ker Q=(\ker Q_y)_{y\in Y}$ and $\ker\hat{Q}$ are isomorphic and then
\[
\ind_a(B)=[\ker Q]-[Y\times \C^q]=[\ker\hat Q]-[Y\times\C^q]=\ind_a(\hat{B}),
\]
as we wanted.
\cqd

\section{Nontrivial bundles}\label{sec:non-trivial_bundles}

In this section we discuss families of Boutet de Monvel operators acting between vector bundles. The case considered 
in the first two sections correspond to the case of trivial bundles over the manifolds and the zero bundle over 
the boundary.  

In addition to the data assumed up to this point (a bundle of manifolds
$\pi\colon Z\to Y$ with fiber $X$), we take smooth vector bundles $E$ and $F$
over $X$ and $\partial X$, respectively. 
Let $\diff(\partial X,F)$ denote the group of diffeomorphisms of $F$ which map fibers to fibers linearly, 
and let $G_E$ denote the group of diffeomorphisms of $E$ which map fibers to fibers linearly and 
whose restrictions to the base belong to the group $G$ defined on page~\pageref{defG}. 
We equip $\diff(\partial X,F)$ with its canonical 
topology \cite[page 123]{AS} and do a similar construction for $G_E$. Note
that there are homomorphisms ``forget the action in the fiber''
$h_\partial\colon \diff(\partial X,F)\to \diff(\partial X)$ and $h\colon
G_E\to G$. Define the fiber product group
\begin{equation*}
G_r:=\{(\phi,\psi)\in
\diff(\partial X,F)\times G_E\mid h_\partial(\phi)=h(\psi)\}.
\end{equation*}
Let $(p\colon
\tilde E\to Z;\; q\colon \tilde F\to\partial Z)$ be maps such that $(\pi\circ
p\colon \tilde E\to Y; \; \pi_\partial\circ q\colon \tilde F\to Y)$ are
bundles with, respectively, 
fibers $E$ and $F$ and structure group $G_r$. It follows
  that, for each pair of local
trivializations $(\alpha,\beta)$ of $(\pi\circ p\colon \tilde E\to Y;\; F\to
Y)$  
there are local trivialization $\alpha_0$ of $\pi\colon Z\to Y$ and $\beta_0$
of $\partial Z\to Y$ such that the diagram 
\begin{equation}
\label{hyp}
\def\mapup#1{\Big\uparrow\rlap{$\vcenter{\hbox{$\scriptstyle#1$}}$}}
\def\mapdn#1{\Big\downarrow\rlap{$\vcenter{\hbox{$\scriptstyle#1$}}$}}
\begin{array}{ccc}
(\pi\circ p)^{-1}(U)&{\mathop{\longrightarrow}\limits^{\alpha}}&U\times E\\
\mapdn{p}&&\Big\downarrow\\
\pi^{-1}(U)&{\mathop{\longrightarrow}\limits^{\alpha_0}}&U\times X
\end{array}
\end{equation}
commutes, where the right vertical arrow is the identity on $U$ times the bundle projection on $E$. 
This defines a vector bundle structure for $p\colon \tilde E\to Z$. Moreover, for each $y\in Y$, the 
restriction of $p$ to $\tilde E_y=(\pi\circ p)^{-1}(y)$
defines a smooth vector bundle $p_y\colon \tilde E_y\to Z_y$, isomorphic to
$E\to X$. We obtain the corresponding result for the the map
$q$ and get a vector bundle $q\colon \tilde F\to \partial Z$ and, 
for each $y\in Y$, a smooth vector bundle $q_y\colon \tilde F_y\to\partial Z_y$ isomorphic to $F\to \partial X$.

Choose now, in addition to the family of Riemannian metrics $(g_y)_{y\in Y}$
families of Hermitean metrics on $E_y$ and $F_y$ which depend continuously on
$y\in Y$. Using them, we get families of Hilbert spaces $H_y:=
L^2(Z_y;E_y)\oplus L^2(\partial Z_y; F_y)$ which patch together to a bundle of
Hilbert spaces. Let $\A(E,F)_y$ denote the C$^*$-subalgebra of the algebra of all bounded operators on 
$H_y$ generated by the polyhomogeneous Boutet de Monvel operators of 
order and class zero.

Exactly as \cite[Proposition 1.3]{AS} our Proposition \ref{prop:conj_is_cont}
generalizes to the case of non-trivial bundles and their diffeomorphisms and
is the basis for the generalization of Corollary
\ref{corol:bundle_of_C_algebras} to the case of non-trivial bundles: the
$\A(E,F)_y$ form in a canonical way a continuous bundle of C$^*$-algebras,
which we continue to call $\aleph$ by abuse of notation.

Let $\mathfrak{A}$ denote the set of continuous sections of the bundle
$\aleph$, forming again a C$^*$-algebra with pointwise operations and supremum
norm. The K-theory results of Section~\ref{sec1}\ can be extended to this
more general setting using arguments similar to those used in \cite{MSS}. In
particular, the analytic and topological index 
given in Section~\ref{index} can also be defined as maps $K_1(\ac)\to
K(Y)$. Theorem~\ref{indthm} then extends to this more general setting. 

\begin{remark}
  Variants of Theorem   \ref{indthm}, the family index theorem for the Boutet
  de Monvel algebra
 for real K-theory or for equivariant K-theory should hold as
  well, and one should be able to derive them along the lines used in the
  present article.
\end{remark}


\begin{appendix}
  \section{Reduction of the structure group}\label{sec:structgroup}

Let, as in the main body of the text, $X$ be a compact smooth manifold with
boundary $\partial X$, and fix a collar diffeomorphism $\delta\colon U\to \partial X\times
[0,1)$ with collar coordinate $x_n$. Recall that $G$ was
defined as the subgroup of the diffeomorphism group $\diff(X)$ of those
diffeomorphisms which respect the product structure and collar coordinate for
$x_n\in [0,1/2)$. For convenience, in the 
text we were working with bundles of 
manifolds modelled on $X$ and with structure group $G$, i.e.~with a
canonically defined collar of the boundary in each fiber of the bundle.

In this appendix, we prove that, for any bundle (over a paracompact space) with structure group
$\diff(X)$ we have a unique (up to isomorphism) reduction to the structure
group $G$. In other words, the functor from bundles (over a given paracompact
base) with structure group $G$
to bundles with structure group $\diff(X)$ which ``forgets the collar'' is an equivalence of categories. [This is similar to the (unique
up to isomorphism) choice of a 
Riemannian metric on a given finite dimensional vector bundle: reduction of
the structure group from $Gl(n)$ to $O(n)$.]

It is well known that we get this unique reduction of structure group if the
inclusion $G\to 
\diff(X)$ is a homotopy equivalence, compare \cite{Do} for a rather refined
version of this fact. We therefore show
\begin{theorem}\label{theo:red_to_G}
  The inclusion $G\to \diff(X)$ (and therefore the corresponding map $BG\to
  B\diff(X)$) 
  are homotopy equivalences. 
\end{theorem}
\begin{proof}
Observe first that $G$ and $\diff(X)$ as well as $BG$ and $B\diff(X)$ are
paracompact 
Fr\'echet manifolds by \cite[Sections 41, 42, 44.21]{KrieglMichor} (the
reference is for $\diff(X)$, but the proofs easily generalize to
$G$). Therefore it
suffices by \cite[Theorem 15]{Palais} to 
show that $G\to \diff(X)$ is a weak homotopy equivalence and it follows
automatically that it is a homotopy equivalence. 

 To show that the map is a weak homotopy equivalence, we have for a 
  continuous map $f\colon K\to\diff(X)$, where $K$ is a compact CW-complex, to
  construct a homotopy $f_s$ from $f_0=f$ to an $f_1$ which takes values in
  $G$. Moreover, the homotopy should be constant on every CW-subcomplex $K_0$
  of $K$ where $f$ already maps to $G$. Note that $K_0$ is a deformation
   retract of a neighborhood $U$, i.e.~there is a homotopy $h
 \colon K\times [0,1]\to K$ from the identiy to $h_1$ such that $h_1(U)=K_0$ and
 such that $h_t$ is the identity on $K_0$. Be precomposing with $h_1$ we can
 therefore assume that $f$ maps the neighbourhood $U$ of $K_0$ to $G$.
 
Let us now construct the family $f_t$.
  Choose $\eta\in (0,1]$ such that $\tilde f(k)=\delta\circ
  f(k)\circ\delta^{-1}$ maps $\partial X\times[0,\eta)$ to $\partial X\times
  [0,1)$ for all $k\in K$ and write $\tilde
  f(k)(x^\prime,t)=(\varphi(x^\prime,t;k),\tau(x^\prime,t;k))$.  

  In two steps we shall now first deform $\tau$ to a function $\hat\tau$ which
  equals $t$ for small $t$ and then $\varphi$ to a function which depends only
  on $x'$ for small $t$.

  Observe that, as $f(k)$ is a diffeomorphism of a manifold with boundary,
  $\frac{\partial \tau}{\partial t}>0$ and therefore, 
  by the compactness of $K$, if we choose $\eta$ small enough,
  $C>\frac{\partial \tau}{\partial t}>c>0$ for some $C>c>0$ on all of
  $K\times \partial X\times [0,\eta)$.

   Pick a smooth function $a\colon [0,\eta)\to [0,1]$ such that $a(t) \equiv
   0$ for $t$ close to zero, $a(t)\equiv1$ for $t$ close to $\eta$ and such
   that
 $$\hat \tau(x',t;k)= (1-a(t))t+a(t)\tau(x',t;k),
 \quad (x',t)\in\partial X\times[0,\eta), $$ 
satisfies $\partial \hat \tau(x',t;k)/\partial t\ge c/2$ for  every $x'\in\partial X$ end every $k\in K$. To construct
 such an $a$, we use the uniform growth of $\tau$: Choose, for some given
  $\varepsilon>0$, the function $a$ so that $(1-a)t$ is monotonely increasing
  on the interval $[0,4\varepsilon]$ with $(1-a)t=t$ on $[0, \varepsilon]$ and
  $(1-a)t=2\varepsilon$ on $[3\varepsilon,4\varepsilon]$.  Then $a$ is
  necessarily increasing with $a\equiv 0$ near $0$ and $a(4\varepsilon)=1/2$.
  Moreover, $\hat\tau$ is strictly increasing as $\tau$ is.  Finally choose
  $a$ on $[4\varepsilon, \eta]$ such that $(1-a)t$ monotonely decreases to $0$
  and equals zero on $[\eta-\varepsilon, \eta]$.  Moreover, we arrange for the
  derivative $\partial_t((1-a)t)$ to be always $\ge
  -2\frac{2\varepsilon}{\eta-5\varepsilon}$.  Again, $a$ is necessarily
  increasing with $a\equiv 1$ near $\eta$.  The derivative $\partial_t(a\tau)$
  can therefore be estimated from below by $c/2$.  For $\varepsilon$
  sufficiently small, we will therefore have
  $2\frac{2\varepsilon}{\eta-5\varepsilon}<c $ and thus
  $\partial_t\hat\tau(x',t;k)>0$ for all $x',t,k$.   
 Note that
then $\hat\tau(x',t;k) = t $ for $t$ close to zero, and $\hat
\tau(x',t;k) = 
\tau(x',t;k)$ for $t$ close to $\eta$, uniformly in $k$.  
We
then let
$$\tau_s= s\hat\tau+(1-s) \tau,\quad 0\le s\le 1.$$
Then $\frac{\partial \tau_s}{\partial t}\ge c/2$ on $K\times \partial X\times
[0,\eta)$. 

For the second step fix a smooth function $\rho\colon [0,1)\to [0,1)$
with $\rho(t) =0$ for $t<\varepsilon$ and $\rho(t) = t$ for
$t>1-\varepsilon$.  Next choose a smooth family of smooth functions
$\rho_s$, $0\le s\le1$ such that $\rho_0$ is the identity and
$\rho_1=\rho$. By compactness, we have a uniform bound $|d\rho_s(t)/dt|\le R$.
For a given $\eta>0$, define $\rho^\eta_s(t)\colon [0,\eta)\to [0,\eta);
t\mapsto \eta\rho_s(\eta^{-1} t)$. Then still $|d\rho_s^\eta/dt|\le R$, even
independently of $\eta$.

Let $\varphi_s(x',t):=\varphi(x',\rho^\eta_s(t))$ and $\tilde f_s(k)(x^\prime,t)=
(\varphi_s(x^\prime,t),\tau_s(t))$.  Then $ \tilde f_s$ equals the given
$\tilde f$ for $t$ close to $\eta$. Therefore $f_s= \delta^{-1}\circ \tilde
f_s\circ \delta$ extends (independently of $s$) to a self-map of
$X$. Moreover, $|{\frac{\partial\tau_s}{\partial x'}}|\le 
|{\frac{\partial \tau}{\partial x'}}|$ for all $s$.  And for $t=0$ we have
  $\frac{\partial \tau}{\partial x'}=0$.
On the other hand, $\frac{\partial \rho_s}{\partial
  x'}|_{(x',t)}=\frac{\partial \rho}{\partial x'}|_{(x',\rho_s(t))}$ is, for
$\eta$ small enough,
 invertible on $[0,\eta]$ with uniform bound on the norm of the inverse (and
 with better bounds if we choose $\eta$ smaller), and $|\frac{\partial
   \phi_s}{\partial t}(x',t)|= |\frac{\partial 
  \phi}{\partial t}_{(x',\rho_s(t))}|\cdot |d\rho^\eta_s/dt(t)|$ which is uniformly
bounded, independent of $\eta$. 
 
    By choosing $\eta$ small enough, 
therefore $\partial\tau_s$ will be linearly independent from $\partial\varphi(x',\rho_s(t))$
and so $f_s(k)$ is a submersion for all $s,k$. 
We check that we actually constructed diffeomorphisms. We made our
construction such that all the maps $f_s(k)$ are submersions which map the
boundary to itself, therefore the image is an open subset of $X$. As $X$ is
compact, the image is also closed, and the map being a local diffeomorphism,
is a covering map. Because it is homotopic to the diffeomorphism $f(k)$, it is
a trivial covering map and therefore a diffeomorphism.

It is obvious that $f_0=f$ and $f_1(k)$ lies in the variant of $G$ where $1/2$
is replaced by $\eta-\epsilon$.

Next,  we compose 
with a family of reparametrizations of the collar $[0,1)$ which stretches
$[0,\eta-\epsilon)$
to $[0,1/2)$ such that in the end we really map to $G$. Note that our
construction is carried out in such a way that for $k\in U$, where $f(k)$ was
already in $G$, $f_s(k)\in G$ for all $s$, although, because of the last
reparametrization step, not necessarily $f_s(k)=f(k)$. 

Therefore, finally, we choose a function $\beta\colon K\to [0,1]$ which is $1$ outside
$U$ and $0$ on $K_0$ and replace the homotopy $f_s(k)$ with $f_{\beta(k)s}(k)$.

This yields the desired homotopy from $f_0=f$ to an $f_1$ taking values in
$G$.  Moreover, the mapping is constant on $K_0$.

\end{proof}


\section{The K\"unneth formula}\label{sec:kunneth}
By the ``K\"unneth formula'', we mean the following theorem of Schochet
\cite{S}:

\begin{thm} \label{KS} Let A and B be C$^*$-algebras with $A$ in the smallest
  subcategory of the category of separable nuclear C$^*$-algebras which
  contains the separable Type I algebras and is closed under the operations of
  taking ideals, quotients, extensions, inductive limits, stable isomorphism,
  and crossed product by $\Z$ and by $\R$. Then there is a natural
  $\Z/2$-graded exact sequence
  \begin{equation}\label{KAS}
    0\to K_*(A)\otimes K_*(B)\to K_*(A\otimes B)\to\mbox{{\em Tor}}(K_*(A),K_*(B))\to 0.
  \end{equation}
\end{thm}

We use this Theorem to prove a statement made in the proof of Theorem
\ref{biso}:

\begin{pro}\label{kun} $b_{U*}\colon K_i(C_0(\partial Z_U))
  \to K_i(\im \gamma_U)$ is an isomorphism, $i=0,1$.
\end{pro}
\pf Let $A=C_0(U)$
and $B=C(\partial X)$.  Then $\im \gamma_U$
is equal to $A\otimes C$, where $C$ is the image of the boundary principal
symbol for the single manifold $X$.  As explained in the Introduction of
\cite{MNS}, $C$ can be regarded as a C$^*$-subalgebra of $C(S^*\partial
X)\otimes\mathcal{T}$, where $\mathcal{T}$ denotes the Toeplitz algebra.
Since $\mathcal{T}$ belongs to the category defined in the statement of
Theorem \ref{KS} (see Examples 5.6.4 and 6.5.1 in \cite{M}), we may apply
Schochet's theorem for $A\otimes B$ and for $A\otimes C$.

Now let ${\tt b}\colon C(\partial X)\to C$ be the map analogous to the map $b$
defined right before the statement of Theorem~\ref{biso}. In
\cite[Section 3]{MNS}, it is proven that ${\tt b}$ induces a
K-theory isomorphism ({\tt b} was denoted $b$ in \cite{MNS,MSS}). Using that
the exact sequence of Theorem~\ref{KS} is natural, we can map (\ref{KAS}) to
the corresponding sequence obtained by replacing $B$ with $C$. Since the maps
induced by ${\tt b}$ are isomorphisms, it follows from the five-lemma that the
maps induced by $b_U=\mbox{id}_A\otimes{\tt b}$ are also isomorphisms. \cqd
\end{appendix}

\section*{Acknowledgements}
We greatly benefited from numerous discussions with our friends 
Johannes Aastrup and Daniel Tausk. 
We thank them for their generosity and for the great time we had talking Math to them. We are also grateful to Jochen Ditsche for pointing out Proposition 
\ref{ditsche} to us. 
Severino Melo was partially supported by a grant from the Brazilian agency
CNPq (Processo 304783/2009-9). Thomas Schick was partially supported by the
Courant Center ``Higher order structures of mathematics'' within the
Excellence initiative's Institutional strategy of Georg-August-Universit\"at G\"ottingen.



\begin{thebibliography}{99}

\bibitem{AM}{\sc P. Albin \&\ R. Melrose}. 
{\em Relative Chern character, boundaries and index formulas}. 
J. Topol. Anal. {\bf 1} (2009), no. 3, 207--250. 

\bibitem{A}{\sc M. F. Atiyah}. $K$-Theory, Lecture notes by D. W. Anderson. 
W. A. Benjamin, Inc., New York-Amsterdam, 1967.  

\bibitem{AS}{\sc M. F. Atiyah \&\ I. M. Singer}. 
{\em  The index of elliptic operators IV}. 
Ann. of Math. (2)  {\bf 93} (1971), 119--138. 

\bibitem{Bl}{\sc B. Blackadar}. 
$K$-Theory for Operator Algebras. 
Cambridge University Press, Cambridge, 1998. 

\bibitem{B}{\sc L. Boutet de Monvel}. 
{\em Boundary problems for pseudo-differential operators}. 
Acta Math. {\bf 126} (1971), no. 1-2, 11--51.


\bibitem{Do} {\sc A. Dold}. {\em Partitions of unity in the theory of fibrations}. Ann. of Math. (2) {\bf 78} (1963), 223--255.

\bibitem{Goh} {\sc I. C. Gohberg}. {\em On the theory of multidimensional singular integral operators}. Dokl. Akad. Nauk SSSR {\bf 133} (1960), 1279--1282.

\bibitem{G}{\sc G. Grubb}. 
Functional Calculus of Pseudodifferential Boundary Problems, Second Edition. 
Birkh\"auser, Boston, 1996.

\bibitem{GG} {\sc G.\ Grubb \&\ G.\ Geymonat}.
{\em The essential spectrum of elliptic systems of mixed order}. Math. Ann. {\bf 227} (1977), 247--276.

\bibitem{Hatcher} {\sc A.\ Hatcher}.
Algebraic topology.  Cambridge University Press, Cambridge, 2002.

\bibitem{J}{\sc K. J\"anich}. 
{\em Vektorraumb\"undel und der Raum der Fredholm-Operatoren}. Math. Ann. {\bf 161} (1965), 129--142. 

\bibitem{Karoubi} {\sc M. Karoubi}. $K$-theory. An introduction. Grundlehren
  der Mathematischen Wissenschaften, Band 226. Springer-Verlag, Berlin-New
  York, 1978.

\bibitem{KN} {\sc J. J. Kohn \&\ L. Nirenberg}. {\em An algebra of
    pseudo-differential operators}. Comm. Pure Appl. Math. {\bf 18} (1965),
  269--305.

\bibitem{KrieglMichor} {\sc A. Kriegl \&\ P. W. Michor}. The convenient
  setting of global analysis. Mathematical Surveys and Monographs {\bf
    53}. American Mathematical Society, Providence, RI, 1997.

\bibitem{MS}{\sc R. Matthes \&\ W. Szymanski}. Lecture Notes on the $K$-Theory of Operator Algebras,
www.impan.gov.pl/Manuals/K\_theory.pdf, 2007.

\bibitem{MNS}{\sc S. T. Melo, R. Nest, E. Schrohe}. 
{\em C$^*$-structure and $K$-theory of Boutet de Monvel's algebra}. 
J. reine angew. Math. {\bf 561} (2003), 145--175.

\bibitem{MSS}{\sc S. T. Melo, T. Schick, E. Schrohe}. 
{\em A $K$-theoretic proof of Boutet de Monvel's index theorem 
for boundary value problems}. 
J. reine angew. Math. {\bf 599} (2006), 217--233.

\bibitem{M}{\sc G. J. Murphy}. C$^*$-Algebras and Operator Theory. 
Academic Press, Boston, 1990.

\bibitem{Palais}{\sc R. S. Palais}.
{\em Homotopy theory of infinite dimensional manifolds}.
Topology {\bf 5} (1966), 1--16.

\bibitem{RS}{\sc S. Rempel \&\ B.-W. Schulze}. 
Index Theory of Elliptic Boundary Problems. Akademie Verlag, Berlin, 1982.

\bibitem{S}{\sc C. Schochet}. 
{\em Topological methods for C$^*$-algebras II: geometric
resolutions and the K\"unneth formula}. 
Pacific J. Math. {\bf 98} (1982), no. 2, 443--458.

\bibitem{S3} {\sc E.~Schrohe}. {\em A short introduction to Boutet de Monvel's calculus}; 
Approaches to Singular Analysis (Berlin, 1999). 
Oper. Theory Adv. Appl. {\bf 125}, 85-116, Birkh{\"a}user, Basel, 2001.

\bibitem{s91} {\sc B.-W.~Schulze}. 
{\em Pseudo-Differential Operators on Manifolds with Singularities}. North Holland, Amsterdam, 1991.

\bibitem{See} {\sc R. T. Seeley}. {\em Integro-differential operators on vector bundles}. Trans. Amer. Math. Soc. {\bf 117} (1965), 167--204.


\end{thebibliography}
\end{document}